\documentclass[a4paper,reqno]{amsart}

\usepackage{amsmath,amsthm,amssymb}
\usepackage{mathrsfs}
\usepackage[shortlabels]{enumitem}
\usepackage{graphicx}
\usepackage[font=small]{caption}
\usepackage{xcolor}
\usepackage{url}

\newcommand{\N}{\mathbb{N}}
\newcommand{\R}{{\mathbb{R}}}
\newcommand{\C}{{\mathbb{C}}}

\newcommand{\dd}{{{\rm d}}}
\newcommand{\ii}{{\rm i}}

\newcommand{\cf}{\emph{cf.}}
\newcommand{\ie}{{\emph{i.e.}}}
\newcommand{\eg}{{\emph{e.g.}}}

\newcommand{\la}{\lambda}

\newcommand{\eps}{\varepsilon}

\newcommand{\Dom}{{\operatorname{Dom}}}

\renewcommand{\Re}{\operatorname{Re}}
\renewcommand{\Im}{\operatorname{Im}}

\newcommand{\sgn}{\operatorname{sgn}}
\newcommand{\supp}{\operatorname{supp}}

\newcommand{\BigO}{\mathcal{O}}

\newcommand{\CcR}{{C_0^{\infty}(\R)}}

\newcommand{\ls}{\lesssim}
\newcommand{\gs}{\gtrsim}

\theoremstyle{plain}

\newtheorem{theorem}{Theorem}[section]
\newtheorem{lemma}[theorem]{Lemma}

\theoremstyle{definition}
\newtheorem{example}[theorem]{Example}
\newtheorem{remark}[theorem]{Remark}
\newtheorem{asm}{Assumption}

\newcommand\cH{\mathcal H}

\newcommand\cL{\mathcal L}

\newcommand\cW{\mathcal W}

\makeatletter
\def\set@curr@file#1{%
	\begingroup
	\escapechar\m@ne
	\xdef\@curr@file{\expandafter\string\csname #1\endcsname}%
	\endgroup
}
\def\quote@name#1{"\quote@@name#1\@gobble""}
\def\quote@@name#1"{#1\quote@@name}
\def\unquote@name#1{\quote@@name#1\@gobble"}
\makeatother
\usepackage{graphics}

\newcommand{\Cla}{\C \setminus (-\infty,0]}

\usepackage{mathtools}
\mathtoolsset{showonlyrefs}

\numberwithin{equation}{section}

\begin{document}

\title[Pseudospectra of the DWE with unbounded damping]{Pseudospectra of the damped wave equation with unbounded damping}

\author{Ahmet Arifoski}
\address[Ahmet Arifoski]{Kantonsschule Olten, Hardfeldstrasse 53, 4600 Olten, Switzerland}
\email{ahmet.arifoski@kantiolten.ch}
\author{Petr Siegl}
\address[Petr Siegl]{
	School of Mathematics and Physics, Queen's University Belfast, University Road, Belfast, BT7 1NN, UK}
\email{p.siegl@qub.ac.uk}

\thanks{The research of P.S. and in particular work on this paper was supported by the \emph{Swiss National Science Foundation} Ambizione grant No.~PZ00P2\_154786 (until December 2017) and by FCT  (Portugal) project PTDC/MAT-CAL/4334/2014. We would like to thank Mark Embree for numerical investigations of pseudospectra carried out during the NOSEVOL Workshop in Berder in July 2013 which provided first insights.}



\date{November 20, 2019}

\begin{abstract}
We analyze pseudospectra of the generator of the damped wave equation with unbounded damping. We show that the resolvent norm diverges as $\Re z \to - \infty$. The highly non-normal character of the operator is a robust effect preserved even when a strong potential is added. Consequently, spectral instabilities and other related pseudospectral effects are present.
\end{abstract}

\maketitle

\section{Introduction}

We consider a linear damped wave equation
\begin{equation}\label{DWE.def}
\partial_t^2 u(t,x)+2a(x) \partial_t u(t,x)=\left(\partial_x^2-q(x)\right) u(t,x), \quad t>0,\quad x\in \R,
\end{equation}
with a non-negative damping $a$ that is unbounded at infinity and a non-negative potential $q$ that is also possibly unbounded. As demonstrated in recent works \cite{Freitas-2018-264, Ikehata-2018}, new effects occur due the unboundedness of $a$. In particular, the new spectral features investigated in \cite{Freitas-2018-264} is the ``overdamping at infinity'' reflected in the presence of the essential spectrum $(-\infty,0]$ responsible for the loss of an exponential energy decay of solutions; for polynomial decay estimates see~\cite{Ikehata-2018}. This paper deals with a more subtle pseudospectral analysis, which reveals highly non-normal character of the semigroup generator $G$, see \eqref{DWE.system}, similarly as for Schr\"odinger operators with complex potentials, see \eg~\cite{Davies-1999-200,Dencker-2004-57,Krejcirik-2015-56,Krejcirik-2019-276}. 
Thus this type of the ``overdamping'' is responsible also for strong spectral instabilities.

Traditionally, the second order wave equation~\eqref{DWE.def} can be rewritten as the first order system 
\begin{equation}\label{DWE.system}
\partial_t 
\begin{pmatrix}
u \\
v
\end{pmatrix}
=
\underbrace{
	\begin{pmatrix}
	0 & I \\
	\partial_x^2 - q & - 2 a  
	\end{pmatrix}}_{G}
\begin{pmatrix}
u \\
v
\end{pmatrix}
\end{equation}
so that the semigroup theory can be employed. Indeed, it was established in \cite{Freitas-2018-264} that, under suitable regularity assumptions on $a$ and $q$, the operator $G$ generates a contraction semigroup in a natural Hilbert space 
\begin{align} 
\cH:=\mathcal{W}( \R )\times L^2( \R ), 
\quad 
\left\langle 
\begin{pmatrix}\psi_1  \\  \psi_2 \end{pmatrix},\begin{pmatrix}\phi_1  \\  \phi_2 \end{pmatrix}\right\rangle_{\cH}
=
\langle\phi_1',\psi_1'\rangle+\langle q^{\frac{1}{2}}\phi_1,q^{\frac{1}{2}}\psi_1\rangle+\langle\phi_2,\psi_2\rangle,
\end{align}
where $\langle \cdot, \cdot \rangle$ and $\|\cdot\|$ are $L^2(\R)$ inner product and norm, respectively, and $\cW(\R)$ is the completion of the pre-Hilbert space $(C_{0}^{\infty}(\R),(\|\partial_x\cdot\|^2+\|q^{\frac{1}{2}}\cdot \|^2)^{\frac{1}{2}})$. Further estimates on the energy decay of solutions for unbounded damping in higher dimensions were performed  in \cite{Ikehata-2018} and polynomial rates established.

The main goal of this work is to study pseudospectra of $G$, aiming at lower estimates of $\|(G-\la)^{-1}\|_{\cL(\cH,\cH)}$ for $\la$ in the left complex half-plane, where the numerical range of $G$ is located. More precisely, we identify curves $\Gamma$ in the left complex half-plane and construct pseudomodes on them, \ie~families $\{\Psi_\la \in \Dom(G) \, : \, \la \in \Gamma\}$ of compactly supported smooth functions satisfying
\begin{equation}\label{G.Omega}
\lim_{\substack{\la \to \infty \\ \la \in \Gamma}} \frac{\|(G-\la)\Psi_\la\|_\cH}{\|\Psi_\la\|_\cH} = 0.
\end{equation}
Our results provide estimates on the decay rates in \eqref{G.Omega} as $\la \to \infty$, which can be turned into lower estimates of the resolvent norm, as well as a description of admissible curves $\Gamma$. Both the rates and curves depend on growth and regularity of $a$ and $q$, see Theorem~\ref{Thm:exp} and examples in Sections~\ref{subsec:ex1} and \ref{subsec:ex2}; see also Remark~\ref{rem:reg}.

In the special case of a monomial damping $a(x) = x^{2m}$, $m \in \N$, and no potential, \ie~$q=0$, a corollary of Theorem~\ref{Thm:exp} reads 
\begin{theorem}\label{thm:intro}
Let $G$ be as in \eqref{DWE.system} where (with some $m\in\N$)
\begin{equation}
q(x)=0, \qquad a(x)=x^{2m}.
\end{equation}
Then, for every $N \in \N$, there exists a $\la$-dependent family of functions $\{\Psi_{\la}\} \subset C_{0}^{\infty}(\R)^2$ such that 
\begin{equation}\label{G.la.est.intro}
\|(G-\la)\Psi_\la \|_{\cH}  = \BigO(|\Re \la|^{-N}) \|\Psi_\la \|_{\cH}
\end{equation}
as $\Re \la \to -\infty$, provided that $\Im \la$ satisfies
\begin{equation}\label{Im.la.intro}
|\Im \la| \geq |\Re \la|^{-\frac{1}{2m} + \eps}, 
\end{equation}
for some $\eps >0$ independent of $\la$.

\end{theorem}
The decay estimate in \eqref{G.la.est.intro} can be made more precise, taking into account also the size of $\Im \la$, and the potential $q$ can be added, which may affect the estimate on the shape of the pseudospectral region $\Omega$ if $q$ is much stronger at infinity than $a$, see Examples~\ref{Ex:pol.1}, \ref{Ex:pol.2} for details. 
\begin{figure}[htb!]
	\includegraphics[width= 0.6 \textwidth]{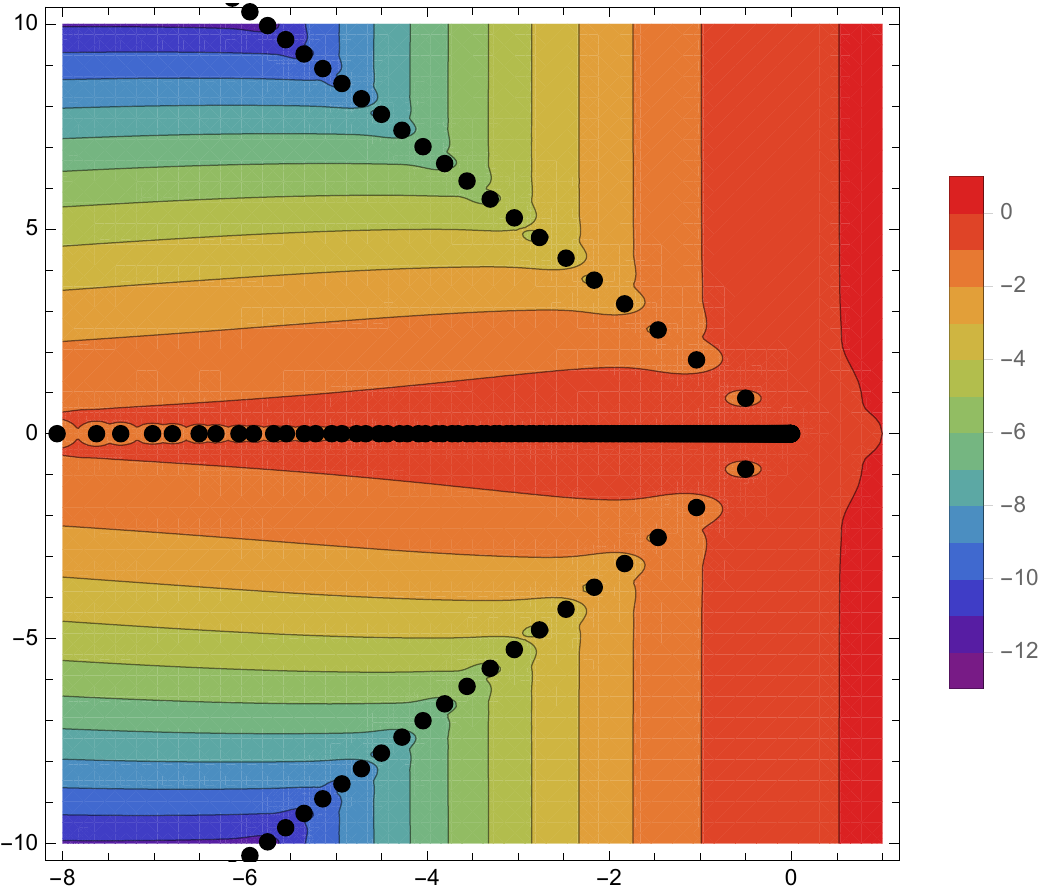}
	\caption{Numerical computation of the spectrum (in black, see \eqref{sp.x^2}) and the pseudospectra ($\log_{10}$ scale, approximation by an $800\times800$ matrix) of $G$ with $a(x)=x^2$ and $q(x)=0$, $x\in \R$.}
	\label{fig:x2}
\end{figure}

Figure \ref{fig:x2} is an illustration of Theorem~\ref{thm:intro} for $a(x) =x^2$ and $q(x)=0$. In this case, the spectrum is explicit (as well as for other monomial dampings, see \cite[Prop.~6.1]{Freitas-2018-264}), namely
\begin{equation}\label{sp.x^2}
\sigma(G) = (-\infty,0] \ \dot \cup \ \left\{ 2^{\frac{1}{3}} e^{ \pm \ii \frac{2}{3} \pi} (2k+1)^{\frac{2}{3}} \right\}_{k \in \N_0 }.
\end{equation}
Clearly, pseudospectra do not localize at a neighborhood of the discrete spectrum and the shape of the pseudospectral region $\Omega$ seems to be in agreement with the analytic result, see \eqref{Im.la.intro}. We note that although the spectrum of this example (also for $a(x) = x^{2m}$) is contained in a sector, in fact lying on three rays, see \eqref{sp.x^2}, the generated semigroup cannot be analytic due to the ``bad'' resolvent behavior in the left-complex half-plane, see \eg~\cite[Sec.~8.4]{Davies-2007}. Moreover, unlike the essential spectrum $(-\infty,0] $ that can be shifted to the left (even made to disappear completely) by adding the potential $q$ of a comparable strength to $a$ at infinity, see \cite[Sec.~7]{Freitas-2018-264}, a vast pseudospectral region is preserved even if the potential $q$ strongly dominates the damping $a$, see Theorems~\ref{Thm:bas}, \ref{Thm:exp} and examples in Sections~\ref{subsec:ex1}, \ref{subsec:ex2}. 

Pseudospectral behavior, more precisely the decay rates and the shape of pseudospectral region $\Omega$, is well-studied for Schr\"odinger operators with complex potentials $V$,
\begin{equation}
H = -\partial_x^2 + V(x),
\end{equation}
see~\cite{Krejcirik-2019-276}; results for semi-classical operators can be found for instance in~\cite{Davies-1999-200,Dencker-2004-57}. The occurrence of the non-trivial pseudospectral region $\Omega$ here is due to the \emph{imaginary part} of $V$ and it essentially depends on its behavior and size at infinity comparing to $\Re V$, see \cite{Krejcirik-2019-276}. 
Our problem originates in equation \eqref{DWE.def} with real coefficients. Nevertheless, the associated spectral problem for $G$ is closely related with the quadratic operator function having the form of a Schr\"odinger operator
\begin{equation}\label{T.def}
T(\la)=-\partial_x^2+q+2\la a+\la^2, \quad \la \in \Cla
\end{equation}
with a ``$\la$-dependent potential'' $q+2\la a+\la^2$ which is \emph{complex} in general. 
In more detail, the spectral equivalence between $G$ and $T$
\begin{equation}
\la\in\sigma(G) \iff 0\in\sigma(T(\la)), \qquad \la \in \Cla,
\end{equation}
was established in a quite general setting with possibly unbounded and singular coefficients, \cf~\cite{Freitas-2018-264}. On a formal level, such equivalence is immediate for eigenvalues since one can check that 
\begin{equation}\label{GT.EF}
(G-\la)(\psi, \phi)^t =0 \iff \phi = \la \psi \quad \text{and} \quad T(\la) \psi =0.
\end{equation}
In fact, the link between $G$ and $T$ is used also for the pseudospectral analysis here. We search for the pseudomodes of $G$ in the form 
\begin{equation}\label{Psi.la.def}
\Psi_\la:=(f_\la,\la f_\la)^t,
\end{equation}
which reduces the problem to the pseudomodes for $T$ since then
\begin{equation}\label{G.T.Psi}
\dfrac{\|(G-\la)\Psi_\la \|_{\cH}^2}{\| \Psi_\la \|_{\cH}^2}
=
\dfrac{\| T(\la)f_\la \|^2}{\| f_\la'\|^2+\| q^{\frac{1}{2}}f_\la\|^2+|\la|^2\| f_\la \|^2}.
\end{equation}
Thus the construction of pseudomodes is based on a complex WKB-method applied for the quadratic operator function $T(\la)$, \cf~\eqref{T.def}. Technically, this goes beyond the simpler semi-classical setting, \cf~for instance~\cite{Davies-1999-200}, due to different powers of $\la$, and the non-semiclassical one for Schr\"odinger operators, \cf~\cite{Krejcirik-2019-276}, due to the ``$\la$-dependent potential'' $q+2 \la a+ \la^2$. Since more than a local behavior of the damping $a$ is required in our approach, the spirit of the performed estimates is closer to the non-semiclassical case, nonetheless, the quadratic dependence on $\la$ in the ``potential'' brings new obstacles and effects. In particular, already basic pseudomodes with one term in the phase, see~\eqref{f.basic}, yield decay in \eqref{G.T.Psi} for various behaviors of $a$ although it is possibly slower and for a smaller region $\Omega$, see examples in Sections~\ref{subsec:ex1}, \ref{subsec:ex2}. Intuitively, this difference to Schr\"odinger operators occurs due to the large parameter $\la$ in front of the damping $a$. 

To achieve better accessibility of the results and proofs, we first construct pseudomodes with a simple basic ansatz, \cf~Section~\ref{sec:basic}, and give examples where these results can be applied, \cf~Section~\ref{subsec:ex1}. In the second step, we employ an ansatz with a phase expansion, \cf~Section~\ref{sec:exp} and demonstrate on examples, \cf~Section~\ref{subsec:ex2}, the improvements for the decay rates and pseudospectral region $\Omega$ comparing to the first step. 

Finally, although we do not give explicit claims here, we remark that our method can be adapted in a straightforward way also to problems with a singular damping like $a(x)=1/x^\alpha$, $\alpha >1$, $x\in (0,1)$. The pseudomodes in this case are constructed with shrinking supports strictly inside $(0,1)$ and so the results are local and independent of possible boundary conditions at $0$. Naturally, our results have also straightforward corollaries for higher dimensional problems allowing for separation of variables and their perturbations; we omit a discussion and explicit claims on these as well.

\subsection{Notations}
We fix some notations used throughout the paper. For positive and negative real numbers we write $ \R _+:=(0,\infty)$ and $ \R _-:=(-\infty,0)$, respectively. To have a short expression for an "integer interval" we use the double brackets, $[[m,n]]:=[m,n]\cap \mathbb{Z}.$ 
Given an interval $I \subset  \R $, the norm of $L^{p}(I)$ is denoted by $\|\cdot\|_{p,I}$. If $I= \R $ we abbreviate $\|\cdot\|_{p,\R}=\|\cdot\|_{p}$; for the most frequent $L^2( \R )$-norm we use $\|\cdot\|:=\|\cdot\|_{2}$. 
To avoid many appearing constants, we employ the convention that $a\ls b$ if there exists a constant $C>0$, independent of any relevant variable or parameter, such that $a\leq Cb$; the relation $a\gs b$ is introduced analogously. By $a\approx b$ it is meant that $a\ls b$ and $b\gs a$.

\section{Basic ansatz}
\label{sec:basic}

We recall that the pseudomodes are constructed for the associated quadratic operator function $T(\la)$, \cf~\eqref{T.def}, having the form of a Schr\"odinger operator with a $\la$-dependent potential. Following this point of view, it is natural that the most basic ansatz for the pseudomode has the form
\begin{equation}\label{f.basic}
f_{\la}:=\xi g, \qquad 
g(x)=\exp \left(\pm \ii\int_{b}^{x}\sqrt{-\la^2-2\la a(t)-q(t)} \, \dd t\right)
\end{equation}
with a suitably chosen $\la$-dependent point $b$, see \eqref{la.par}, \eqref{b.def}, and a $\la$-dependent cut-off  $\xi \in \CcR$, see \cite{Krejcirik-2019-276} for details on Schr\"odinger operators. 

Due to the accretivity of $-G$, see \eg~\cite{Freitas-2018-264} for details, the construction is relevant only for $\la$ in the left complex half-plane, \ie~in the second and third quadrant. It suffices to analyze the second quadrant as $a$ and $q$ are assumed to be real and the third quadrant can be reached by complex conjugation. Hence, we parametrize $\la$ in the following way
\begin{equation}\label{la.par}
\la=-\alpha+\ii\beta,  \qquad \alpha > 0, \quad \beta > 0
\end{equation}
and we choose $-\ii$ in the formula for $g$ in \eqref{f.basic}. This choice guarantees that the principal complex square root, \ie~for $z\in\mathbb{C}\setminus \R _{-}$ we take
\begin{equation} \label{square-root}
\sqrt{z}= \sqrt{\dfrac{|z|+\Re(z)}{2}} + \sgn(\Im(z)) \ii \sqrt{\dfrac{|z|-\Re(z)}{2}},
\end{equation} 
which is always used here, is continuous since, for $x \in \supp \xi$, the values of $-\la^2 - 2 \la a -q$ stay away from $\R_-$, see \eqref{zeta.def} and \eqref{Re.zeta}. 

Up to the cut-off $\xi$, the basic ansatz has the form $e^{h(x)}$ with a complex valued function $h$ which we want to behave essentially as $-(x-b)^2$ with $b \to +\infty$. To this end, we search for the point $b\in \R_+$ such that $\Re h(b)=0$ and $\Re h$ does not change the sign at $0$. It will be showed that the suitable choice is given by the equation
\begin{equation}\label{b.def}
\alpha=a(b),
\end{equation}  
see \eqref{Re.Im.zeta}. The cut-off function $\xi$ is chosen such that 
\begin{equation} \label{xi.def}
\begin{aligned}
&\xi\in C_{0}^{\infty}( \R ), \quad \ 0\leq\xi\leq 1, &\\&
\xi(b+s)=1, \quad  |s|<\dfrac{\delta}{2}, &\\&
\xi(b+s)=0, \quad  |s|>\delta.
\end{aligned}
\end{equation}
Notice that $\xi$ can be selected in such a way that
\begin{equation}\label{xi'.est}
\|\xi^{(j)}\|_{\infty}\ls\delta^{-j}, \quad j=1,2.
\end{equation}

For notational convenience in further sections and also as a preparation for the ansatz with the phase expansion, we introduce a function $\psi_{-1}$ by
\begin{equation}\label{psi-1.bas}
\la\psi_{-1}(x)=\int_{b}^{x}\sqrt{-\la^2-2\la a(t)-q(t)} \, \dd t, \quad x\in\supp(\xi)
\end{equation}
and so $f_\la$ can be written as
\begin{equation}\label{f.la.bas}
f_\la := \xi g = \xi \exp(- \ii \la \psi_{-1}).
\end{equation}

\subsection{The main result for the basic ansatz}

In this section, we work under the following assumption. We construct pseudomodes supported in $\R_+$ and use the behavior of $a$ and $q$ in $\R_+$ only. The construction can be clearly repeated in $\R_-$ if the assumptions on $a$ and $q$ are adjusted accordingly.

\begin{asm}\label{Asm:bas}
Suppose that the functions $a \in C^2(\R)$, $q\in C^{1}( \R )$ satisfy the following conditions:
\begin{enumerate}[(a)]
\item $a$, $q$ are non-negative for sufficiently large $x$:
\begin{equation}
\forall x\gs 1, \quad a(x)\geq0 \quad \text{and} \quad q(x) \geq 0;
\end{equation}
\item $a$ is increasing for sufficiently large $x$ and unbounded at infinity:
\begin{align}
&\forall x\gs 1, \quad a'(x) > 0, \label{asm:bas.a.incr}
\\
&\lim_{x\to+\infty}a(x)=+\infty; \label{asm:bas.a.unbd}
\end{align}
\item the derivatives of $a$ are controlled by $a$:
\begin{equation}
\exists \nu \geq-1, \quad \forall x \gs 1, \quad |a^{(j)}(x)| \ls x^{j\nu }a(x), \quad j=1,2.
\label{asm:bas.a'}
\end{equation}
\end{enumerate}
\end{asm}

We recall that the condition \eqref{asm:bas.a'} guarantees almost a constant behavior of functions $a$ and $a'$ on sufficiently small intervals, namely if $\Delta= o(x^{-\nu})$ as $ x \to +\infty$, then
\begin{equation}\label{a.const}
\dfrac{a^{(j)}(x+\Delta)}{a^{(j)}(x)}\approx1, \quad j =0,1, \quad x \to +\infty;
\end{equation}
the detailed proof can be found \eg~in \cite[Sec.~3]{Krejcirik-2019-276}. 

The result on pseudomodes is formulated for curves in the left-complex half-plane with $\Re \la = - \alpha \to - \infty$. As $b$ is defined via the equation $\alpha=a(b)$, see \eqref{b.def}, by assumptions \eqref{asm:bas.a.incr} and \eqref{asm:bas.a.unbd}, we can write $b \to +\infty$  instead of $\alpha \to +\infty$.

\begin{theorem}\label{Thm:bas}
Let Assumption~\ref{Asm:bas} hold and let $b\in \R _+$ be defined by $\alpha = a(b)$. Take $\eps>0$ and define 
\begin{equation}\label{Thm:bas.asm.delta}
\delta:= b^{-\nu - \eps} 
\end{equation}
and
\begin{equation}\label{q.j.def.1}
q^{(j)}_b := \| q^{(j)}\|_{\infty,(b-\delta,b+\delta)}, \quad j = 0,1.
\end{equation}

Suppose that there exists a $b$-dependent $\beta=\beta(b)>0$ such that the following conditions hold as $b \to + \infty$
\begin{align}
\label{Thm:bas.asm.cutoff}
\forall c>0, \quad
\left(\frac{1}{\delta} + \frac1{\alpha+\beta} \frac1{\delta^2}\right)\exp \left(-c\frac{\beta}{\alpha+\beta} a'(b) \delta^2 \right)&=o(1), 
\\
\label{Thm:bas.asm.q0}
q_b^{(0)} &=o(\alpha^2+\beta^2),
\\
\label{Thm:bas.asm.aq}
\alpha^2 b^{2 \nu} + q_b^{(1)}
&=o(\alpha^2+\beta^2). 
\end{align}
Let $\{f_\la\}$ with $\la= \la(b) = -\alpha + \ii \beta$, see \eqref{la.par}, 
be a family of functions constructed as in \eqref{f.la.bas} with $\xi$ as in \eqref{xi.def} and the choice of $\alpha$, $\beta$ and $\delta$ as above.

Then, for $\Psi_\la = (f_\la,\la f_\la)^t$, we have
\begin{equation}\label{Thm:bas.claim}
\dfrac{\|(G-\la)\Psi_\la\|_{\cH}}{\|\Psi_\la\|_{\cH}} = o(1), \quad b \to + \infty.
\end{equation}
\end{theorem}

The conditions in Theorem~\eqref{Thm:bas} have a complicated structure since they combine several competing terms with a different origin. In detail, the size of the cut-off, \ie~the choice of $\delta$, the shape of the curve along which $\la$ tends to infinity, \ie~the choice of $\beta$, the growth of the damping, \ie~the size of $a'(b)$, and the size and growth of the potential $q$, \ie~the size of $q$ and $q'$, play against each other. 

The conditions~\eqref{Thm:bas.asm.q0}, \eqref{Thm:bas.asm.aq} guarantee that $q$ can be treated as a small perturbations in the estimates and that the remainder, see \eqref{R.bas.def} below, is small. The condition \eqref{Thm:bas.asm.cutoff} ensures that the cut-off $\xi$ produces small terms. For simplicity, we assume that the limit in \eqref{Thm:bas.asm.cutoff} is zero for all $c>0$, nevertheless, it suffices to verify this only for one constant $c=c_2/8$ appearing in the proof. However, to estimate the size of $c_2$, more precise information on $a'$ would be needed, see Lemma~\ref{Lem:Re.psi.-1.est}. As we show in examples, see Section \ref{subsec:ex1}, also the stronger condition \eqref{Thm:bas.asm.cutoff} can be verified easily. Notice also that \eqref{Thm:bas.asm.cutoff} is always satisfied if $\nu<0$ and when a sufficiently small $\eps>0$ is chosen, thus \eg~for polynomial-like functions with $\nu=-1$.

\subsection{Strategy and technical lemmas}

We recall that due to \eqref{G.T.Psi}, we need to estimate $\|T(\la) f_\la\|$.  
With regard to our ansatz for $f_{\la}$, see \eqref{f.la.bas}, \eqref{psi-1.bas}, we arrive at
\begin{equation}\label{T.f.la}
\begin{aligned}
T(\la)f_{\la}&=-\xi''g-2\xi'g'-\xi g''+(2\la a+q+\la^2)\xi g&
\\
&
=-\xi''g+2\xi' \ii \la\psi_{-1}'g+\xi \left[\ii \la\psi_{-1}''+\la^2(\psi_{-1}')^2 g+2\la a+q+\la^2\right] g.
\end{aligned}
\end{equation}
Recall that we aim to having the function $g$ essentially as a Gaussian concentrated around $b$. Notice that the first two terms in \eqref{T.f.la} contain derivatives of the cut-off $\xi$ for which we clearly have
\begin{equation}\label{xi'.supp}
\supp(\xi^{(j)})\subset[b-\delta,b-\delta/2]\cup[b+\delta/2,b+\delta], \quad j \geq 1. 
\end{equation}
As $g$ is constructed such that $|g|$ is exponentially localized around $b$, the terms $\|\xi''g\|^2 $ and $\|\xi'\la\psi_{-1}'g\|^2$, both divided by $|\la|^2\|\xi g \|^2$, are expected to be small as $b \to +\infty$. Also the term $\|\xi \la\psi_{-1}''g\|^2$ is expected to be small when divided by $|\la|^2\|\xi g \|^2$ due to the assumption \eqref{asm:bas.a'}. However, there is no reason why the remaining terms
in \eqref{T.f.la} should be small and so we impose this by requiring 
\begin{equation}\label{zeta.def}
\zeta:=\la^2(\psi_{-1}')^2=-\la^2-2\la a-q.
\end{equation} 
Hence the remainder to be controlled reads
\begin{equation}\label{R.bas.def}
R:= \ii \xi \la \psi_{-1}''.
\end{equation}

As mentioned before, the idea is to find a point $b\in \R $ such that $\Im \zeta=0$, see \eqref{psi-1.bas}, \eqref{f.la.bas} and \eqref{square-root}. To this end, observe that 
\begin{equation}\label{Re.Im.zeta}
\Re \zeta =-\alpha^2+\beta^2+2\alpha a -q,
\qquad 
\Im\zeta =2\beta(\alpha-a),
\end{equation}
which leads to the choice of $b$ as in \eqref{b.def}, \ie~$\alpha=a(b)$.

It is important to notice that, using \eqref{Thm:bas.asm.q0}, we obtain
\begin{equation}
\Re \zeta(b)=-\alpha^2+\beta^2+2\alpha a(b) -q(b)=\alpha^2+\beta^2-q(b)\approx \alpha^2+\beta^2>0.
\end{equation}
The next lemma shows that this remains true also in the $\delta$-neighborhood of $b$, more precisely for all $t\in \supp(\xi)$. 

\begin{lemma}\label{Lem:Re.zeta}
Let Assumption \ref{Asm:bas} hold, let $b\in\mathbb{R_+}$ be defined by the equation \eqref{b.def}, let $\eps>0$, $\delta$ be as in \eqref{Thm:bas.asm.delta} and $\zeta$ be as in \eqref{zeta.def}. Suppose that $q$ satisfy \eqref{Thm:bas.asm.q0}, \ie~
\begin{equation}\label{asm:bas.sum.dq}
q^{(0)}_b=o(\alpha^2+\beta^2), \quad b \to + \infty.
\end{equation}
Then, for all $t\in(b-\delta,b+\delta)$, we have as $b \to + \infty$ that 
\begin{equation}\label{Re.zeta}
\Re\zeta (t)\approx \alpha^2+\beta^2.	
\end{equation}
\end{lemma}
\begin{proof}
We give a detailed proof for $t>b$ only, the case $t<b$ is very similar.
By the mean value theorem, the assumption \eqref{asm:bas.a'}, the property \eqref{a.const} and the choice of $\delta$, see~\eqref{Thm:bas.asm.delta}, we get for each $t \in (b,b+\delta)$ that
\begin{equation}\label{atb.est}
a(t)-a(b)=a'(\eta)(t-b) 
\ls
a(b) b^\nu \delta
=
o(a(b)), \quad b \to +\infty. 
\end{equation}
Using \eqref{atb.est}, \eqref{a.const}, \eqref{b.def} and the assumption on $q$, see \eqref{asm:bas.sum.dq}, for all $t\in(b,b+\delta)$, we get that 
\begin{equation}
\begin{aligned}
\Re \zeta (t)&= \alpha(a(t)+a(t)-\alpha)+\beta^2-q(t)
\approx \alpha^2 (1-o(1))+\beta^2-q(t)
\approx \alpha^2+\beta^2
\end{aligned}
\end{equation}
as $b \to +\infty$.
\end{proof}

Next, we aim to get two-sided estimates of $|g|$, for which we have to estimate $\Re(\ii \la\psi_{-1})$. As a preparation we have the following lemma.
\begin{lemma}\label{Lem:zeta.est}
Let the assumptions of Lemma~\ref{Lem:Re.zeta} hold. Then for all $t\in(b-\delta,b+\delta)$ 
\begin{equation}\label{|zeta|.est}
(\Re \zeta (t))^{\frac{1}{2}}+|\Im \zeta (t)|^{\frac{1}{2}}\approx \alpha+\beta, \quad b \to +\infty.
\end{equation}
\end{lemma}
\begin{proof}
We give a detailed proof for the case $t> b$ only,  the case $t< b $ is very similar. 
By Lemma \ref{Lem:Re.zeta} we already have 
\begin{equation}
(\Re \zeta (t))^{\frac{1}{2}}\approx \alpha+\beta, \quad b \to + \infty.
\end{equation}

We proceed with the estimate of $|\Im \zeta(t)|^{\frac{1}{2}}$, see \eqref{Re.Im.zeta}. From \eqref{atb.est} and \eqref{b.def}, we obtain (recall that we consider only $\beta>0$ here)
\begin{equation}
|\Im\zeta (t)|=2\beta|a(b)-a(t)| = o(\alpha^2+\beta^2), 
\quad b \to +\infty. 
\end{equation}
Hence,
\begin{equation}
(\Re \zeta(t))^{\frac{1}{2}}+|\Im\zeta(t)|^{\frac{1}{2}}
\approx
(\alpha+\beta)(1+ o(\alpha + \beta))
\approx \alpha+\beta, \quad b \to +\infty.
\end{equation}
\end{proof}
The next step is the two-sided estimate of $\Re(\ii \la\psi_{-1})$.
\begin{lemma}\label{Lem:Re.psi.-1.est}
Let the assumptions of Lemma~\ref{Lem:Re.zeta} hold. Then there exist two positive constants $c_1, c_2 >0$ such that for all $s\in(b-\delta,b+\delta)$
\begin{equation}\label{Re.psi-1.est}
c_2\dfrac{\beta}{\alpha+\beta}a'(b)(s-b)^2\leq \Re(\ii\la\psi_{-1}(s))\leq c_1\dfrac{\beta}{\alpha+\beta}a'(b)(s-b)^2, \quad b \to +\infty.
\end{equation}
\end{lemma}
\begin{proof}
Rewriting $\Re(\ii\la\psi_{-1})(s)$ and using \eqref{square-root}, we obtain
\begin{equation}
\begin{aligned}
\Re(\ii\la\psi_{-1})(s)
&=
\int_{b}^{s}\Re(\ii\sqrt{\zeta(t)}) \, \dd t
=
-\int_{b}^{s}\Im \sqrt{\zeta(t)} \, \dd t
\\
&=
-\int_{b}^{s} \sgn(\Im\zeta(t)) \sqrt{\dfrac{|\zeta(t)|-\Re \zeta(t)}{2}}\, \dd t.
\end{aligned}
\end{equation}
Taking into account \eqref{asm:bas.a.incr} and \eqref{Re.Im.zeta}, we observe that (recall that $\beta>0$)
\begin{equation}
\sgn(\Im\zeta(t))=\sgn(2\beta (a(b)-a(t))) = \sgn(b-t).
\end{equation}

Further we consider the case $s>b$ only, the other one is fully analogous. By straightforward manipulations and estimates, we obtain (recall that $\Re \zeta(t)>0$) 
\begin{equation}
\begin{aligned}
\Re(\ii\la\psi_{-1})(s) 
&\approx
\int_{b}^{s}\sqrt{\dfrac{|\zeta(t)|^2-(\Re \zeta(t))^2}{|\zeta(t)|+\Re \zeta(t)}} \, \dd t
=
\int_{b}^{s}\dfrac{|\Im \zeta(t)|}{\sqrt{|\zeta(t)|+\Re \zeta(t)}} \, \dd t
\\
&
\approx  \int_{b}^{s}\dfrac{|\Im \zeta(t)|}{(\Re\zeta(t))^{\frac{1}{2}}+|\Im \zeta(t)|^\frac{1}{2}} \, \dd t, \quad s>b.
\end{aligned}
\end{equation}
By \eqref{Re.Im.zeta}, Lemma \ref{Lem:zeta.est}, the mean value theorem and \eqref{a.const}, we get ($\beta$, $a'(t) > 0$)
\begin{equation}
\begin{aligned}
\Re(\ii \la\psi_{-1})(s) & 
\approx \int_{b}^{s}\dfrac{|\beta(a(b)-a(t))|}{\alpha+\beta} \, \dd t\approx \int_{b}^{s}\dfrac{|\beta a'(b)(b-t)|}{\alpha+\beta} \, \dd t
\\
&
\approx \dfrac{\beta}{\alpha+\beta}a'(b)\int_{b}^{s}(t-b) \, \dd t
\approx \dfrac{\beta}{\alpha+\beta}a'(b)(s-b)^2,
\end{aligned}
\end{equation}
hence \eqref{Re.psi-1.est} holds for $s>b$.
\end{proof}

The next step is to estimate $\| \xi g \|^2$ from below and $ \| \xi' g' \|^2$, $\| \xi'' g \|^2$ from above. 
\begin{lemma}\label{Lem:bas.xi.xi'g}
Let the assumptions of Lemma~\ref{Lem:Re.zeta} hold, let $\xi$, $g$ be as in \eqref{xi.def},  \eqref{f.la.bas}, respectively, and let $c_2>0$ be as in Lemma~\ref{Lem:Re.psi.-1.est}. Then, as $b \to +\infty$,  
\begin{align}\label{xig eq}
\| \xi g\|^2 &\gs \delta 
\exp \left(
- \frac{c_2}{4} \frac{\beta}{\alpha+\beta} a'(b) \delta^2
\right),
\\
\label{xi'g' eq}
\| \xi' g'\|^2 &\ls (\alpha^2+\beta^2)\delta^{-1} \exp\left({-\frac{c_2}{2} \frac{\beta}{\alpha+\beta}a'(b)\delta^2}\right),
\\
\label{xi''g eq}
\| \xi'' g\|^2&\ls \delta^{-3} \exp\left({-\frac{c_2}{2}\frac{\beta}{\alpha+\beta}a'(b)\delta^2}\right).
\end{align} 
\end{lemma}
\begin{proof}
We start with the estimate of $\| \xi g \|^2$.
Using the definition of $\xi$, see \eqref{xi.def}, 
\begin{equation}
\| \xi g \|^2 = \int_{b-\delta}^{b+\delta}\big| \xi(s)\big|^2\big|e^{-\ii\la\psi_{-1}(s)} \big|^2 \, \dd s
\geq 
\int_{b-\frac{\delta}{2}}^{b+\frac{\delta}{2}}e^{-2\Re(\ii\la\psi_{-1}(s))} \, \dd s.
\end{equation}
So with the upper bound of $\Re(\ii \la\psi_{-1}(s))$ in Lemma \ref{Lem:Re.psi.-1.est}, we get
\begin{equation}
\| \xi g \|^2 \geq 
\int_{b-\frac{\delta}{2}}^{b+\frac{\delta}{2}} e^{-2c_1\frac{\beta}{\alpha+\beta} a'(b)(s-b)^2} \, \dd s=\int_{-\frac{\delta}{2}}^{\frac{\delta}{2}} e^{-2 c_1\frac{\beta}{\alpha+\beta} a'(b)s^2} \, \dd s.
\end{equation}
Since $c_1 \geq c_2$, taking a positive $k$ such that $k^2:=2c_1 /c_2 >1$, we arrive at
\begin{equation}
\| \xi g \|^2
\geq \int_{-\frac{\delta}{2k}}^{\frac{\delta}{2k}} e^{-2 c_1\frac{\beta}{\alpha+\beta}a'(b)s^2} \, \dd s 
\gs \delta e^{-2 c_1\frac{\beta}{\alpha+\beta}a'(b)\frac{\delta^2}{4k^2}}
= \delta e^{- \frac{c_2}{4} \frac{\beta}{\alpha+\beta}a'(b) \delta^2},
\end{equation}
thus \eqref{xig eq} is proved.

Next we analyze $\|\xi'' g\|^2$. The estimate \eqref{xi'.est} and the lower bound for $\Re(\ii \la\psi_{-1}(s))$ in \eqref{Re.psi-1.est} lead to
\begin{equation}
\| \xi'' g\|^2
= \int_{\mathbb{\supp(\xi'')}}^{}\big| \xi''(s)\big|^2 e^{-2\Re(\ii \la\psi_{-1}(s))}  \, \dd s
\ls\delta^{-4}\int_{\mathbb{\supp(\xi'')}}^{} e^{-2c_2\frac{\beta}{\alpha+\beta}a'(b)(s-b)^2}  \, \dd s.
\end{equation}
By symmetry and \eqref{xi'.supp}, we get
\begin{equation}
\| \xi'' g\|^2
\ls 
\delta^{-4}\int_{\delta/2}^{\delta} e^{-2 c_2\frac{\beta}{\alpha+\beta}a'(b)s^2}  \, \dd s
\ls
\delta^{-3} e^{-\frac{c_2}{2} \frac{\beta}{\alpha+\beta}a'(b)\delta^2},
\end{equation}
so \eqref{xi''g eq} is proved.

Finally, to estimate $\| \xi' g'\|^2 $, we use \eqref{zeta.def}, \eqref{|zeta|.est} and obtain
\begin{equation}
\| \xi' g'\|^2 
\approx 
\int_{ \R } |\xi'(s)|^2 |\la\psi_{-1}'(s)|^2 |e^{-\ii\la\psi_{-1}(s)} |^2 \, \dd s
\ls \frac{\alpha^2 + \beta^2}{\delta^2} \int_{\supp (\xi')}  |e^{-\ii\la\psi_{-1}(s)} |^2 \, \dd s,
\end{equation}
which can be continued as the estimate of $\| \xi'' g\|^2$.
\end{proof}

It remains to estimate the remainder $R$, see \eqref{R.bas.def}. 

\begin{lemma}\label{Lem:bas.R}
Let the assumptions of Lemma~\ref{Lem:Re.zeta} hold and let $R$ be as in \eqref{R.bas.def}.
Then
\begin{equation}\label{R.bas.est}
\| R\|_{\infty,(b-\delta,b+\delta)}^{2}
\ls \alpha^2 b^{2\nu} + \frac{(q^{(1)}_b)^2}{\alpha^2+\beta^2}, \quad b \to +\infty.
\end{equation}
\end{lemma}
\begin{proof}
From \eqref{zeta.def} and \eqref{|zeta|.est}, we have for all $t\in(b-\delta,b+\delta)$ that
\begin{equation}
|\la\psi_{-1}''(t)|=\left|\dfrac{-2\la a'(t)-q'(t)}{2(-\la^2-2\la a(t)-q(t))^{\frac{1}{2}}}\right|
\approx
\dfrac{|\la| \big|2a'(t)+\frac{q'(t)}{\la}\big|}{|\zeta|^{\frac{1}{2}}}
\ls a'(t)+\frac{|q'(t)|}{\alpha+\beta}.
\end{equation}
Since
\begin{equation}
\| R\|_{\infty,(b-\delta,b+\delta)}^{2}
\leq 
\|\la\psi_{-1}''\|_{\infty,(b-\delta,b+\delta)}^{2},
\end{equation}  
the claim follows from \eqref{asm:bas.a'} and \eqref{q.j.def.1}. 
\end{proof}

\subsection{The proof of Theorem \ref{Thm:bas}} 
\label{subsec:Thm.bas.proof}
Equipped with Lemmas~\ref{Lem:bas.xi.xi'g} and \ref{Lem:bas.R}, we are in position to prove the main result of this section.  

\begin{proof}[Proof of Theorem~\ref{Thm:bas}]
Using \eqref{T.f.la}, the triangle inequality and the definition of the remainder $R$, see 
\eqref{R.bas.def}, we have
\begin{equation}
\frac{\| T(\la)\xi g\|^2}{\|(\xi g)' \|^2 +\|q^\frac12 \xi g \|^2 + |\la|^2\|\xi g\|^2}
\ls 
\underbrace{\dfrac{\|\xi''g\|^2}{|\la|^2\|\xi g\|^2}}_{=:Q_1}+\underbrace{\dfrac{\|\xi'g'\|^2}{|\la|^2\|\xi g\|^2}}_{=:Q_2}+\underbrace{\dfrac{\| R\|_{\infty,(b-\delta,b+\delta)}^{2}\|\xi g\|^2}{|\la|^2\|\xi g\|^2}}_{=:Q_3},
\end{equation}
where we kept only $|\la|^2\|\xi g\|^2$ in the denominator in the last step; $q$ might be $0$ and it can be showed that the term $\|(\xi g)'\|^2$ does not improve the estimate in general.

We start with $Q_1$. By \eqref{xig eq} and \eqref{xi''g eq}, see Lemma~\ref{Lem:bas.xi.xi'g}, we arrive at 
\begin{equation}
\begin{aligned}
Q_1&
\ls
\dfrac{\delta^{-3} e^{-\frac{c_2}{2} \frac{\beta}{\alpha+\beta} a'(b) \delta^2}} {(\alpha^2+\beta^2) \delta  e^{-\frac{c_2}{4} \frac{\beta}{\alpha+\beta}a'(b) \delta^2}}
=
\dfrac{1}{(\alpha^2+\beta^2)\delta^4}e^{- \frac{c_2}{4} \frac{\beta}{\alpha+\beta}a'(b)\delta^2}, \quad b \to + \infty.
\end{aligned}
\end{equation}
Hence the assumption \eqref{Thm:bas.asm.cutoff} yields that $Q_1=o(1)$ as $b \to + \infty$.

The estimate of $Q_2$ is similar. From \eqref{xig eq} and \eqref{xi'g' eq}, we get
\begin{equation}
Q_2
\ls
\dfrac{(\alpha^2+\beta^2)\delta^{-1} e^{-\frac{c_2}{2}\frac{\beta}{\alpha+\beta}a'(b)\delta^2}} {(\alpha^2+\beta^2) \delta e^{-\frac{c_2}{4}\frac{\beta}{\alpha+\beta}a'(b)\delta^2}}
=
\dfrac{1}{\delta^2}e^{-\frac{c_2}{4} \frac{\beta}{\alpha+\beta}a'(b)\delta^2}, \quad b \to +\infty,
\end{equation}
hence the assumption \eqref{Thm:bas.asm.cutoff} yields again that $Q_2 = o(1)$ as $b \to + \infty$.

Finally, we estimate $Q_3$ using \eqref{R.bas.est} and the assumption~\eqref{Thm:bas.asm.aq}, namely
\begin{equation}
Q_3
\ls
\dfrac{\alpha^2 b^{2\nu}}{\alpha^2+\beta^2}
+ 
\dfrac{(q^{(1)}_b)^2}{\alpha^4+\beta^4} = o(1), 
\quad b \to +\infty.
\end{equation}
Hence, putting all the estimate above and \eqref{G.T.Psi} together, we get \eqref{Thm:bas.claim}.
\end{proof}

\subsection{Examples}
\label{subsec:ex1}

\begin{example}\label{Ex:pol.1}(Polynomial-like dampings and potentials) 
	
First we consider dampings $a \in C^2(\R)$ and potentials $q \in C^1(\R)$ satisfying Assumption~\ref{Asm:bas} with $\nu=-1$ and 
\begin{equation}\label{a.q.pol.ex.1}
\forall x \gs 1, \quad a(x)=x^p, \quad q(x) \ls x^r, \quad |q'(x)| \ls x^{r-1} 
\quad p,r\in\mathbb{R_{+}}.
\end{equation}
We determine $b$ when $b \to +\infty$ from the equation $\alpha=a(b)$, see \eqref{b.def}, namely,
\begin{equation}
b=\alpha^{\frac{1}{p}}, \quad b \to + \infty.
\end{equation}
For a sufficiently small $\eps>0$, we take $\delta=b^{1-\eps}$ and start to check the conditions in Theorem \ref{Thm:bas}. To this end, we observe that as $b \to + \infty$, we have
\begin{equation}\label{evaluations poly}
a(b) = b^p, \quad q_b^{(j)}  \ls b^{r-j}, \quad j=0,1. 
\end{equation}
The condition \eqref{Thm:bas.asm.cutoff} guaranteeing the successful cut-off is clearly satisfied independently of the choice of $\beta(b)$ since $\delta \to \infty$ as $b \to + \infty$. To satisfy the remaining conditions~\eqref{Thm:bas.asm.q0} and \eqref{Thm:bas.asm.aq}, we impose the following restrictions on $\beta(b)$
\begin{equation}\label{Ex.pol.beta.1}
\begin{aligned}
&\text{if} \quad r\geq 2p, \qquad \beta(b) \gs b^{s}, \quad s>\frac{r}{2},
\\&\text{if} \quad r< 2p, \qquad \beta(b)>0.
\end{aligned}
\end{equation} 
Recalling \eqref{evaluations poly} and our choice of $\beta$, we indeed have
\begin{equation}
\frac{\alpha^2}{b^2} + q^{(0)}_b  +q^{(1)}_b \ls b^{2p-2} +b^r + b^{r-1}=o(\alpha^2+\beta^2), \quad b \to +\infty.
\end{equation} 

In summary, with the choice of $\beta=\beta(b)$ in \eqref{Ex.pol.beta.1}, the statement of Theorem~\ref{Thm:bas} holds.
\end{example}

\begin{example}\label{Ex:exp.1}(Exponential dampings and potentials) 
	
Next we consider dampings $a \in C^2(\R)$ and potentials $q \in C^1(\R)$ satisfying 
\begin{equation}\label{a.q.exp.ex.1}
\forall x \gs 1, \quad a(x)=e^{x^p}, \quad q(x) \ls  e^{x^r}, \quad |q'(x)| \ls x^{r-1} e^{x^r}, \quad p,r\in\R_{+}, 
\end{equation}
thus Assumption~\ref{Asm:bas} holds with $\nu=p-1$. We further suppose that
\begin{equation}\label{Ex.exp.1.rp}
\quad r \leq p.
\end{equation}
From \eqref{b.def}, we have 
\begin{equation}
b= (\ln\alpha)^{\frac{1}{p}}, \quad b \to + \infty.
\end{equation}
With a sufficiently small $\eps>0$, we take $\delta=b^{-(p-1)-\eps}$ and obtain that as $b \to + \infty$
\begin{equation}
\label{evaluations exponential}
a(b)=e^{b^p},  
\quad
a'(b) \approx b^{p-1}e^{b^p},
\quad 
q_b^{(j)} \ls e^{(1+o(1))b^r}, \quad j =0,1. 
\end{equation}
With regard to the conditions \eqref{Thm:bas.asm.q0} and \eqref{Thm:bas.asm.aq}, it follows from \eqref{Ex.exp.1.rp} that 
\begin{equation}
q_b^{(0)} + q_b^{(1)} \ls e^{(1+o(1))b^r} = o(\alpha^2), \quad b \to + \infty,   
\end{equation}
thus no restrictions on $\beta$ are imposed. On the other hand, the first term in \eqref{Thm:bas.asm.aq} behaves as
\begin{equation}
a(b)^2 b^{2\nu} = b^{2(p-1)}e^{2b^p}, \quad b \to + \infty, 
\end{equation}
thus we obtain the following restrictions on $\beta$
\begin{equation}\label{Ex.exp.beta.1}
\begin{aligned}
&\text{if} \quad p \geq 1, \qquad \beta(b) \gs b^s e^{b^p}, \quad s> p-1,
\\&\text{if} \quad p<1, \qquad \beta(b)>0.
\end{aligned}
\end{equation} 

Finally, we can verify that with this choice of $\beta$, also the condition \eqref{Thm:bas.asm.cutoff} is satisfied. Indeed, for every $c>0$, we have
\begin{equation}\label{beide 1 exp}
\left(
 b^{2(p-1)+2\eps}+\frac{b^{4(p-1)+4\eps}}{\alpha^2+\beta^2}\right)
 e^{-c\frac{\beta}{\alpha+\beta}a'(b)b^{-2(p-1)-2\eps}} = o(1), \quad b \to + \infty
\end{equation}
since, in the non-obvious case $p \geq 1$, we have at least exponential decay due to
\begin{equation}\label{check example exponential}
\frac{\beta}{\alpha+\beta}a'(b)b^{-2(p-1)-2\eps}
\approx 
b^{-p+1-2\eps} e^{b^p}, \quad b \to + \infty.
\end{equation}

In summary, with $\beta=\beta(b)$ as in \eqref{Ex.exp.beta.1}, the statement of Theorem~\ref{Thm:bas} holds.

\end{example}

\begin{example}\label{Ex:log.1}(Logarithmic dampings and potentials)
	 
Finally, we consider dampings and potentials $a \in C^2(\R)$ and $q\in C^1(\R)$ satisfying
\begin{equation}\label{a und q in logarithmus beispiel}
\forall x \gs 1, \quad a(x)=\ln(x), \quad q(x) + |q'(x)| \ls \ln(x),
\end{equation}
which satisfy Assumption \ref{Asm:bas} with $\nu=-1$. From \eqref{b.def}, we immediately have that
\begin{equation}
b=e^{\alpha}, \quad b \to + \infty
\end{equation}
and for a sufficiently small $\eps>0$, we take $\delta=b^{1-\eps}$. It follows that, as $b \to +\infty$,
\begin{equation}\label{evaluations logarithm}
a(b)=\ln(b), \quad a'(b)=\frac{1}{b}, \quad q_b^{(j)} \ls \ln(b), \quad j=0,1. 
\end{equation} 
Hence the condition \eqref{Thm:bas.asm.aq} holds without any restrictions on $\beta>0$ as
\begin{equation}
\frac{a(b)^2}{b^2} + q_b^{(0)} + q_b^{(1)} \ls \ln (b) = o(\alpha^2), \quad b \to + \infty.
\end{equation}
Since $\delta \to +\infty$ as $b \to + \infty$, the condition \eqref{Thm:bas.asm.cutoff} holds also without any restrictions on $\beta$. 

In summary, the statement of Theorem \ref{Thm:bas} holds with any choice of $\beta(b)>0$.
\end{example}

\section{Expansion of the phase}
\label{sec:exp}

In Section~\ref{sec:basic}, only the simplest form of the pseudomode was used. In detail, the function $g$ in \eqref{f.la.bas} has the form
\begin{equation}
g=\exp\left( -\ii\la\psi_{-1} \right).
\end{equation} 
If $a$ and $q$ are more regular, more terms in the exponent can be considered, namely we employ a general WKB expansion
\begin{equation}\label{g.exp.def}
g:=\exp \left(-\ii\la\psi_{-1}-\sum_{k=0}^{n-1}\la^{-k}\psi_k\right);
\end{equation}
the functions $\psi_k$ are determined by a standard procedure briefly summarized in Section~\ref{subsec:exp.idea} below.

The basic ansatz in Section \ref{sec:basic} works already for several important examples, see Section~\ref{subsec:ex1}. Nonetheless, by taking more terms in $g$, we obtain faster decay rates in the main statement, see \eqref{Thm:bas.claim}, which we make quantitative this time, see \eqref{Thm:exp.claim} and~\eg~the obtained rates \eqref{rate.pol.exp} in the example with a polynomial damping. The expansion allows also to achieve a larger set of curves along which we have a decay in the main statement \eqref{Thm:exp.claim}. In other words, we can relax restrictions on the choice of $\beta=\beta(b)$, see examples in Sections~\ref{subsec:ex1} and \ref{subsec:ex2}.

As in the previous case, we need to employ a $\la$-dependent cut-off $\xi\in C_{0}^{\infty}( \R )$ to construct a suitable pseudomode $f_\la:=\xi g$; the choice of $\xi$ is the same as in Section~\ref{sec:basic}, see \eqref{xi.def}.

\subsection{The main result} In this section, we assume the following basic regularity and growth assumptions on the damping and potential and state the main result of the paper.

\begin{asm}\label{Asm:exp}
Suppose that the functions $a, q\in C^{n+1}(\R)$ with $n \in \N$ satisfy the following conditions:
\begin{enumerate}[(a)]
\item $a$, $q$ are non-negative for sufficiently large $x$:
\begin{equation}
\forall x\gs 1, \quad a(x)\geq0 \quad \text{and} \quad q(x) \geq 0;
\end{equation}
\item $a$ is increasing for sufficiently large $x$ and unbounded at infinity:
\begin{align}
&\forall x\gs 1, \quad a'(x) > 0, \label{asm:exp.a.incr}
\\
&\lim_{x\to+\infty}a(x)=+\infty; \label{asm:exp.a.unbd}
\end{align}
\item the derivatives of $a$ are controlled by $a$:
\begin{equation}\label{asm:exp.a'}
\exists\nu \geq -1, \quad \forall m\in [[1,n+1]], \quad \forall x \gs 1 \quad |a^{(m)}(x)| \ls  x^{m\nu}a(x).
\end{equation}

\end{enumerate}
\end{asm}

\begin{theorem}\label{Thm:exp}
Let Assumption~\ref{Asm:exp} hold and let $b\in \R _+$ be defined by $\alpha = a(b)$. Take $\eps>0$ and define 
\begin{equation}\label{Thm:exp.asm.delta}
\delta:= b^{-\nu - \eps} 
\end{equation}
and
\begin{equation}\label{q.j.def.exp}
q^{(j)}_b := \| q^{(j)}\|_{\infty,(b-\delta,b+\delta)}, \quad j \in [[0,n]].
\end{equation}

Suppose that there exists a $b$-dependent $\beta=\beta(b)>0$ such that the following conditions hold as $b \to + \infty$
\begin{align}
\label{kappa1.def}
\forall c>0, \quad
\left(\frac{1}{\delta} + \frac1{\alpha+\beta} \frac1{\delta^2}\right)\exp \left(-c\frac{\beta}{\alpha+\beta} a'(b) \delta^2 \right)&=: \kappa_1(b,c) = o(1), 
\\
\label{Thm:exp.asm.q.0}
q_b^{(0)} &=o(\alpha^2+\beta^2),
\\
\label{Thm:exp.asm.q}
\forall j \in [[1,n]], \quad q_b^{(j)} & = \BigO\left( \alpha (\alpha+\beta)b^{j \nu} \right),
\\
\label{Thm:exp.asm.cutoff-2}
\frac{b^\nu}{\alpha+\beta} &= \BigO(1).
\end{align}
Let $\{f_\la\}$ with $\la= \la(b) = -\alpha + \ii \beta$, see \eqref{la.par}, 
be a family of functions constructed as in \eqref{f.la.bas} with $\xi$ as in \eqref{xi.def}, $g$ as in \eqref{g.exp.def} and the choice of $\alpha$, $\beta$ and $\delta$ as above.

Then there exists a positive $C>0$ such that for $\Psi_\la = (f_\la,\la f_\la)^t$, we have
\begin{equation}\label{Thm:exp.claim}
\frac{\|(G-\la)\Psi_\la\|_{\cH}}{\|\Psi_\la\|_{\cH}} 
\ls \kappa_1(b,C)+\kappa_2(b), \quad b \to + \infty,
\end{equation}
where
\begin{equation}\label{kappa2.def}
\kappa_2(b) := \frac{\alpha b^{\nu(n+1)}}{(\alpha+\beta)^n}
+
\sum_{k=1}^{n-1} \frac{b^{\nu(n+k+1)} \alpha^2}{(\alpha+\beta)^{n+1+k}}.
\end{equation}

\end{theorem}

Analogously to the remarks below Theorem~\ref{Thm:bas}, the condition \eqref{kappa1.def} is actually too strong. It suffices to satisfy it for one sufficiently small constant $C=c_4/8>0$, which also enters $\kappa_1(b,C)$ and which can be estimated with more detailed information on $a'$. As examples show, the decay of $\kappa_1$ is typically much faster than of $\kappa_2$ which then determines the final decay rate estimate.

\begin{remark}\label{rem:reg}
It might appear that the dependence of the rates in \eqref{Thm:exp.claim} as well as the conditions on $\beta$, determining~the curves along which we have a decay in \eqref{Thm:exp.claim}, are just a limitation of the method. However, several examples and results for Schr\"odinger operators, see~\cite{Davies-1999-200,Boulton-2002-47,Pravda-Starov-2006-73,Henry-2017-7,Krejcirik-2019-51}, suggest that this dependence is fundamental and indeed reflecting the regularity of coefficients and their behavior at infinity. A more detailed discussion can be found 
in the introduction in \cite{Krejcirik-2019-51}. 
\end{remark}

\subsection{WKB expansion}
\label{subsec:exp.idea}

We follow the standard WKB procedure, for details see~\eg~\cite{Davies-1999-200} or in particular \cite[Sec.~2.4]{Krejcirik-2019-276} where only minor modifications (mainly notational, one should set $V:=2\la a +q$ and the new spectral parameter is $-\la^2$ instead of $\la$) are needed. We obtain that 
\begin{equation}\label{rn.def}
T(\la)g = \left(\sum_{k=n-1}^{2(n-1)}\la^{-k}\phi_{k+1}\right)g=:r_n g,
\end{equation}
where functions $\phi_k$ are defined as (with some $c_{\omega, \chi} \in \C$ with $|c_{\omega,\chi}|=1$)
\begin{equation}\label{phik.def}
\psi_k''-\sum_{\substack{\omega+\chi=k \\ \omega,\chi \neq-1}}^{}c_{\omega,\chi}\psi_\omega'\psi_\chi'=:\phi_{k+1},
\end{equation}
with the convention that $\psi_\omega=0$ whenever $\omega\geq n$ or $\omega\leq-2$ and $\psi_k$ satisfy
\begin{equation}\label{alle psi'}
\begin{aligned}
\psi_{-1}'&=\Bigg(\dfrac{-\la^2-2\la a-q}{\la^2}\Bigg)^{\frac{1}{2}},\\
\psi_{k+1}'&=\dfrac{1}{2\psi_{-1}'}\Bigg(\psi_k''-\sum_{\substack{\omega+\chi=k \\ \omega,\chi \neq-1}}^{}c_{\omega,\chi}\psi_\omega'\psi_\chi'\Bigg), \quad k\in[[-1,n-2]],
\end{aligned}
\end{equation}
again with the same convention for $\psi_\omega$. For the function $\psi_0'$, one gets in particular
\begin{equation}\label{psi0'.def}
\psi_0'=-\dfrac{1}{4}\dfrac{2\la a'+q'}{-\la^2-2\la a -q}.
\end{equation}

For the forthcoming estimates, it is crucial to understand the structure of the functions $\psi_k'$ and remainders $r_n$, which is the content of the following two lemmas; detailed proofs (with minor modifications in notations) are in \cite[Appendix]{Krejcirik-2019-276}.

\begin{lemma}\label{Lem:T.exp}
Let $ n \in \mathbb{N}_0 $, $ a,q \in C^{n+1}{( \R )} $ and functions $\{\psi'_k\}_{k\in [[-1,n-1]]} $ be determined by \eqref{alle psi'}. Then 
\begin{equation}\label{psik.exp.T}  
\psi_k^{(m)}=\dfrac{(-\la^2)^\frac{k}{2}}{(-\la^2-2\la a-q)^{\frac{k}{2}}}\sum_{j=0}^{k+m}\dfrac{T_j^{k+m,k+m+1-j}}{(-\la^2-2\la a-q)^j}, \quad m \in [[1,n+1-k]],
\end{equation} 
where (with some $c_{\omega} \in \mathbb{C}$)
\begin{equation} \label{Tjrs.def}
T_j^{r,s} := \sum_{\omega\in \mathcal{I}_j^{r,s}} c_{\omega}((2\la a+q)^{(1)})^{\omega_1}((2\la a+q)^{(2)})^{\omega_2}\cdot...\cdot((2\la a+q)^{(s)})^{\omega_s}, 
\end{equation} 
and
\begin{equation}\label{Ijrs.def}
\mathcal{I}_j^{r,s} := \left\{\omega \in \mathbb{N}_0^{s}: \sum_{i=1}^{s} i\omega_i=r \quad \& \quad \sum_{i=1}^{s} \omega_i=j\right\}.
\end{equation}	
\end{lemma}
\begin{lemma}\label{Lem:r.exp}
Let $ n\in \mathbb{N}_{0}$, $a,q\in C^{n+1}( \R ) $  and functions $\{\psi'_k\}_{k\in [[-1,n-1]]} $ be determined by \eqref{alle psi'},  $\{\phi_{k}\}_{k\in[[-1,2n-1]]} $ be as in \eqref{phik.def} and $r_n$ as in \eqref{rn.def}. Then
\begin{equation} \label{Abschatzungr_nwelle}
|r_n| \ls  \dfrac{|(2\la a+q)^{(n+1)}|}{|\la^2+2\la a+q|^{\frac{n+1}{2}}}+\sum_{k=0}^{n-1}\dfrac{1}{|\la^2+2\la a+q|^{\frac{n-1+k}{2}}}\sum_{j=2}^{n+1+k}\dfrac{T_{j}^{n+1+k,n}}{|\la^2+2\la a+q|^{j}},
\end{equation}
where $T_{j}^{r,s}$ are as in \eqref{Tjrs.def}.
\end{lemma}

\subsection{Technical lemmas}
\label{subsec:technical.lem.exp}

The first lemma enables us to treat the terms with $k\geq0$ in the expansion of $g$, see \eqref{g.exp.def}.
\begin{lemma}\label{Lem:exp.est.1}
Let Assumptions \ref{Asm:exp} hold, let $\{\psi'_k\}_{k\in[[-1,n-1]]}$ be determined by \eqref{alle psi'} and let, as $b \to +\infty$,
\begin{equation}\label{exp.asm.q}
\begin{aligned}
q_b^{(0)} &=o(\alpha^2+\beta^2),
\\
\forall j \in [[1,n]], \quad q_b^{(j)} &= \BigO\left( \alpha(\alpha+\beta)b^{j \nu} \right).
\end{aligned}
\end{equation}	
Then, for all $ k\in[[0,n-1]]$ and for all $t\in (b-\delta,b+\delta)$
\begin{equation}\label{la.k.psi.k.est}
\big|\la^{-k}\psi_k'(t)\big| 
\ls
\frac{b^{\nu(k+1)}}{(\alpha+\beta)^k}
\sum_{j=1}^{k+1}
\frac{\alpha^j}{(\alpha+\beta)^{j}}, \quad b \to +\infty.
\end{equation}
\end{lemma}

\begin{proof}
The essential ingredient of the proof is Lemma \ref{Lem:T.exp}. We omit writing the argument $t$, but we always use that $t\in (b-\delta,b+\delta)$. From \eqref{psik.exp.T}, the  definition of $T_j^{r,s}$, see \eqref{Tjrs.def}, and $\big|\la^2+2\la a+q \big|=|\zeta|\approx\alpha^2+\beta^2\approx|\la|^2$, see Lemma \ref{Lem:zeta.est},  we get 
\begin{equation}
\begin{aligned}
\big|\la^{-k}\psi_k'\big| 
&\ls 
\sum_{j=1}^{k+1} \frac{|T_j^{k+m,k+2-j}|}{|\la^2+2\la a+q |^{j+\frac{k}{2}}}
\\
&\ls
\sum_{j=1}^{k+1}
\sum_{\omega\in \mathcal I_j^{k+1,k+2-j}}
\frac{
|2\la a'+q'|^{\omega_{1}}\cdot...\cdot |2\la a^{(k+2-j)} +q^{(k+2-j)} |^{\omega_{k+2-j}}}{|\la|^{2j+k}}.
\end{aligned}
\end{equation}
The assumptions \eqref{asm:exp.a'} and \eqref{exp.asm.q} give further that
\begin{equation}
\big|\la^{-k}\psi_{k}'\big| 
\ls
\sum_{j=1}^{k+1}
\sum_{\omega\in \mathcal I_j^{k+1,k+2-j}}
\frac{|\la \alpha b^\nu|^{\omega_{1}}\cdot...\cdot |\la \alpha b^{\nu(k+2-j)} |^{\omega_{k+2-j}}}{|\la|^{2j+k}}.
\end{equation}
The definition of $\mathcal{I}_j^{r,s}$, see \eqref{Ijrs.def}, yields  $\sum_{i=1}^{k+2-j}i\omega_i=k+1$ and $\sum_{i=1}^{k+2-j}\omega_i=j$, thus \eqref{la.k.psi.k.est} follows.
\end{proof}

The next aim is to estimate $|g|$. It turns out that with the assumptions above the result for the basic pseudomode with $n=0$ remains valid (with possibly different constants), see~Lemma~\ref{Lem:Re.psi.-1.est}.

\begin{lemma}\label{Lem:|g|.exp}
Let the assumptions of Lemma~\ref{Lem:exp.est.1} hold and suppose in addition that
\begin{equation}\label{lem:exp.asm.cutoff-2}
\frac{b^\nu}{\alpha+\beta} = \BigO(1), \quad b \to + \infty.
\end{equation}
Let $g$ be defined as in \eqref{g.exp.def}. Then there exist two positive constants $c_3, c_4 >0$ such that for all $s\in(b-\delta,b+\delta)$,
we have, as $b \to +\infty$,
\begin{equation}\label{est:|g|.exp}
\exp{\left(c_3\dfrac{\beta}{\alpha+\beta}a'(b)(s-b)^2\right)}
\ls |g(s)|\ls
\exp{\left( c_4\dfrac{\beta}{\alpha+\beta}a'(b)(s-b)^2\right)}.
\end{equation}	
\end{lemma}

\begin{proof}
First we deal with the terms with $k > 0$ in the expansion, the case $k=0$ is treated separately and differently. With Lemma \ref{Lem:exp.est.1} and assumption \eqref{lem:exp.asm.cutoff-2} we get 
\begin{equation}\label{exp.bdd}
\Bigg|\sum_{k=1}^{n-1}\int_{b}^{s}\la^{-k}\psi_{k}'(t) \, dt\Bigg| 
\ls 
\sum_{k=1}^{n-1}\sum_{j=1}^{k+1}
\frac{b^{\nu k - \eps}  \alpha^j}{(\alpha + \beta)^{j+k}}
\ls
b^{-\eps}, \quad b \to +\infty.
\end{equation}
In the case $k=0$, we use the formula for $\psi_0'$, see \eqref{psi0'.def}, which leads to
\begin{equation}
|\exp(-\psi_0(s))|=\left|\exp \left(-\int_{b}^{s}\psi_0'(t) \, dt \right) \right| =\Bigg|\dfrac{\la^2+2\la a(s)+q(s)}{\la^2+2\la a(b)+q(b)}  \Bigg|^{\frac{1}{4}}=\left|\dfrac{\zeta(s)}{\zeta(b)}  \right|^{\frac{1}{4}}.
\end{equation}
With Lemma \ref{Lem:zeta.est} and \eqref{zeta.def}, we get that for all $s\in(b-\delta,b+\delta)$
\begin{equation}\label{psi_0.est}
|\exp(-\psi_0(s))|\approx 1, \quad b \to +\infty.
\end{equation}	

Now we are ready to estimate $|g|$. Using \eqref{psi_0.est} and \eqref{exp.bdd}, we get, as $b \to +\infty$,
\begin{equation}
\begin{aligned}
|g| &= \exp \Big(\Re \Big(-\ii\la\psi_{-1}-\sum_{k=0}^{n-1}\la^{-k}\psi_k \Big) \Big)
=  |\exp(-\psi_0 + o(1))| |\exp (\Re (-\ii\la\psi_{-1}))|,
\end{aligned}
\end{equation}
which holds for all $s\in(b-\delta,b+\delta)$. Thus \eqref{est:|g|.exp} follows from Lemma~\ref{Lem:Re.psi.-1.est}.
\end{proof}

The next step is to estimate $\| \xi g \|^2$ from below and $ \| \xi' g' \|^2$, $\| \xi'' g \|^2$ from above. In fact, we show that under our assumptions, these estimates remain the same as for the basic pseudomode, see Lemma~\ref{Lem:exp.xi.xi'g}.

\begin{lemma}\label{Lem:exp.xi.xi'g}
Let the assumptions of Lemma~\ref{Lem:|g|.exp} hold. Let $\xi$, $g$ be defined as in \eqref{xi.def}, \eqref{g.exp.def}, respectively. Then, as $b \to + \infty$,  
\begin{align}\label{xig.eq.exp}
\| \xi g\|^2 &\gs \delta 
\exp \left(
- \frac{c_4}{4} \frac{\beta}{\alpha+\beta} a'(b) \delta^2
\right),
\\
\label{xi'g'.eq.exp}
\| \xi' g'\|^2 &\ls (\alpha^2+\beta^2)\delta^{-1} \exp\left({-\frac{c_4}{2} \frac{\beta}{\alpha+\beta}a'(b)\delta^2}\right),
\\
\label{xi''g.eq.exp}
\| \xi'' g\|^2&\ls \delta^{-3} \exp\left({-\frac{c_4}{2}\frac{\beta}{\alpha+\beta}a'(b)\delta^2}\right).
\end{align} 
\end{lemma}

\begin{proof}
Using the previously proved lemmas, we reduce the proof to the estimates obtained in Lemma~\ref{Lem:bas.xi.xi'g} for the basic pseudomode; with regard to Lemma~\ref{Lem:|g|.exp}, this is immediate for the estimates of $\|\xi g\|$ and $\| \xi'' g\|$. The remaining term $\|\xi' g'\|$ is estimated using Lemmas~\ref{Lem:zeta.est}, \ref{Lem:exp.est.1} and assumption \eqref{lem:exp.asm.cutoff-2}. Recalling \eqref{xi'.est} and the size of $\supp \xi'$, see \eqref{xi'.supp}, we obtain
\begin{equation}
\begin{aligned}
\| \xi' g'\|^2 
&\ls 
\int_{\supp \xi'}\|\xi'\|_{\infty}^2|g'(s)|^2\, \dd s 
\\
&\ls  
\delta^{-2} \int_{\supp \xi'} 
\Big(
\underbrace{|\la\psi'_{-1}(s)|^2}_{=|\zeta|\approx(\alpha+\beta)^2}+\sum_{k=0}^{n-1}| \la^{-k}\psi'_{k}(s)|^2
\Big)
|g(s)|^2 \, \dd s
\\ 
& \ls 
\delta^{-2}\int_{\supp \xi'} 
\Big( 
(\alpha+\beta)^2 + b^{2\nu}
\Big)
|g(s)|^2 \, \dd s, 
\\ 
& \ls 
\delta^{-2}\int_{\supp \xi'} (\alpha+\beta)^2(1+\BigO(1))|g(s)|^2 \, \dd s, \qquad b \to + \infty.
\end{aligned}
\end{equation}
Thus \eqref{xi'g'.eq.exp} follows from Lemma~\ref{Lem:bas.xi.xi'g} as well.
\end{proof}

Finally, we estimate the remainder $r_n$, see \eqref{rn.def}.

\begin{lemma}\label{Lem:rn.est.exp}
Let the assumptions of Lemma~\ref{Lem:|g|.exp} hold and let $r_n$ be as in \eqref{rn.def}.
Then
\begin{equation}\label{rn.est.exp}
\| r_n\|_{\infty,(b-\delta,b+\delta)} 
\ls 
\frac{\alpha b^{\nu(n+1)}}{(\alpha+\beta)^n}
+
\sum_{k=1}^{n-1} \frac{b^{\nu(n+k+1)} \alpha^2}{(\alpha+\beta)^{n+1+k}} = \kappa_2(b), \qquad b \to + \infty.
\end{equation}
\end{lemma}
\begin{proof}
Lemma \ref{Lem:r.exp} on the structure of $r_n$, Lemma \ref{Lem:zeta.est} on the size of $\zeta$ and the assumption \eqref{exp.asm.q} yield 
\begin{equation}\label{rn.est.1}
\begin{aligned}
|r_n| &\ls 
\frac{|(2\la a+q)^{(n+1)}|}{|\la^2+2\la a+q|^{\frac{n+1}{2}}} + \sum_{k=0}^{n-1} \sum_{j=2}^{n+1+k} \frac{|T_{j}^{n+1+k,n}|}{|\la^2+2\la a+q|^{\frac{n-1+k}{2}+j}}
\\
& \ls
\frac{\alpha b^{\nu(n+1)}}{|\la|^n} + \sum_{k=0}^{n-1} \sum_{j=2}^{n+1+k} \frac{|T_{j}^{n+1+k,n}|}{|\la|^{n-1+k+2j}}.
\end{aligned}
\end{equation}

From the definition of $T_{j}^{n+1+k,n}$, see \eqref{Tjrs.def} and \eqref{Ijrs.def}, and assumptions \eqref{asm:exp.a'} and \eqref{exp.asm.q}, we get for all $s \in (b-\delta,b+\delta)$  
\begin{equation}
\begin{aligned}
|T_{j}^{n+1+k,n}(s)| & \ls  
\sum_{\omega\in \mathcal{I}_j^{n+1+k,n}} 
\big|2\la a'(s)+q'(s)\big|^{\omega_1} \cdot...\cdot \big|2\la a^{(n)}(s)+q^{(n)}(s)\big|^{\omega_{n}}
\\
& \ls
\sum_{\omega\in \mathcal{I}_j^{n+1+k,n}} (\alpha |\la| b^\nu)^{\omega_1} \cdot...\cdot (\alpha |\la| b^{\nu n})^{\omega_n}
\\
& \ls
(\alpha |\la|)^j b^{\nu(n+1+k)}.
\end{aligned}
\end{equation} 
Inserting this into \eqref{rn.est.1} and observing that $\alpha/|\la| \approx \alpha/(\alpha+\beta) \leq 1$, which enables us to skip the term with $k=0$ in the sum, we obtain the estimate \eqref{rn.est.exp} in the claim.
\end{proof}

\subsection{The proof of the main Theorem~\ref{Thm:exp}}

Lemmas~\ref{Lem:|g|.exp} and \ref{Lem:rn.est.exp} are analogues of Lemmas \ref{Lem:Re.psi.-1.est} and \ref{Lem:bas.R} and so the proof of the main Theorem~\ref{Thm:exp} becomes a direct analogue of the one of Theorem~\ref{Thm:bas}.

\begin{proof}[Proof of Theorem \ref{Thm:exp}]
As in the proof of Theorem~\ref{Thm:bas}, the estimates are split into three parts. 
Using \eqref{T.f.la} we arrive at (neglecting $\|(\xi g)'\|^2 +\|q^\frac12 \xi g \|^2 $ as before in the proof of Theorem~\ref{Thm:bas})
\begin{equation}
\frac{\| T(\la)\xi g\|^2}{\|(\xi g)' \|^2 +\|q^\frac12 \xi g \|^2 +|\la|^2\|\xi g\|^2} 
\ls  
\underbrace{\dfrac{\|\xi''g\|^2}{|\la|^2\|\xi g\|^2}}_{=:Q_1}+\underbrace{\dfrac{\|\xi'g'\|^2}{|\la|^2\|\xi g\|^2}}_{=:Q_2}+\underbrace{\dfrac{\| r_n\|_{\infty,(b-\delta,b+\delta)}^{2}\|\xi g\|^2}{|\la|^2\|\xi g\|^2}}_{=:Q_3}.
\end{equation}
The terms $Q_1$, $Q_2$, $Q_3$ are estimated in the completely same way as in the proof of Theorem~\ref{Thm:bas}, nevertheless, appropriate replacements of technical steps, \ie~Lemmas~\ref{Lem:|g|.exp} and \ref{Lem:rn.est.exp}, are employed and definitions of $\kappa_1$, $\kappa_2$ are used, see \eqref{kappa1.def}, \eqref{kappa2.def}.
\end{proof}

\subsection{Examples}
\label{subsec:ex2}

\begin{example}\label{Ex:pol.2}(Polynomial-like dampings and potentials  -- continued)
	
We consider non-negative dampings $a \in C^{n+1}(\R)$ and potentials $q \in C^{n+1}(\R)$ with $n>1$   satisfying, similarly as in \eqref{a.q.pol.ex.1},  
\begin{equation}\label{a.q.pol.ex.2}
\forall x \gs 1, \quad a(x)=x^p, \quad |q^{(j)}(x)| \ls x^{r-j} 
\quad p,r\in\mathbb{R_{+}}, \quad j \in [[0,n]].
\end{equation}
It is easy to see that Assumption~\ref{Asm:exp} holds with $\nu=-1$.

Notice that the decay of $\kappa_1(b,c)$, see \eqref{kappa1.def}, is exponential for every $c>0$ if $\beta=\beta(b) \gs b^s$ with $s>-1$, thus the decay rate in \eqref{Thm:exp.claim} comes from the remainder, \ie~from $\kappa_2(b)$. For smaller $\beta$, one needs to compare the decay rates of $\kappa_1$ and $\kappa_2$; we omit discussing such cases in the following. 

Choosing $\beta=\beta(b)$ similarly in \eqref{Ex.pol.beta.1}, namely, 
\begin{equation}\label{Ex.pol.beta.2}
\begin{aligned}
&\text{if} \quad r\geq 2p, \qquad \beta(b) \gs b^{s}, \quad s> r-p,
\\&\text{if} \quad r< 2p, \qquad \beta(b) \gs b^{s}, \quad s > -1,
\end{aligned}
\end{equation} 
we can easily check that all conditions \eqref{Thm:exp.asm.q.0}, \eqref{Thm:exp.asm.q} and \eqref{Thm:exp.asm.cutoff-2} are indeed satisfied. 

Finally, we determine the decay rates of $\kappa_2$ in \eqref{Thm:exp.claim}. Ignoring even the contributions of $\beta$, we get (in fact from the first term in \eqref{kappa2.def})
\begin{equation}\label{rate.pol.exp}
\kappa_2(b) = \BigO( b^{-(n-1)(p+1)-2}), \qquad b \to + \infty.
\end{equation}
\end{example}

\newpage

\begin{example}(Exponential dampings and potentials -- continued).

Similarly as in Example \ref{Ex:exp.1}, we consider non-negative $a,q \in C^{n+1}(\R)$ with $n>1$ of the form
\begin{equation}\label{a.q.exp.ex.2}
\forall x \gs 1, \quad a(x)=e^{x^p}, \quad |q^{(j)}(x)| \ls x^{r-j} e^{x^r}, \quad p,r\in\R_{+}, \ r \leq p, \quad j \in [[0,n]], 
\end{equation}
which satisfy Assumption~\ref{Asm:exp} with $\nu = p-1$. 

It is straightforward to check that the condition~\eqref{kappa1.def} holds if we choose $\beta=\beta(b)$ in the following way
\begin{equation}\label{Ex.exp.beta.2}
\begin{aligned}
&\text{if} \quad p \geq 1, \qquad \beta(b) \gs b^s, \quad s> p-1,
\\&\text{if} \quad p<1, \qquad \beta(b)>0.
\end{aligned}
\end{equation} 
The remaining conditions \eqref{Thm:exp.asm.q.0}, \eqref{Thm:exp.asm.q} and \eqref{Thm:exp.asm.cutoff-2} are satisfied due to $r \leq p$. 

The second term in \eqref{Thm:exp.claim} can be estimated as
\begin{equation}
\kappa_2(b) = \BigO(b^{(p-1)(n+1)} e^{(1-n) b^p}), \qquad b \to + \infty.
\end{equation}
If $\beta$ is small, then the estimate of $\kappa_1(b,C)$ might be lengthier and the rate can be effectively slower than of $\kappa_2(b)$. Nonetheless, for \eg~$\beta \gs \alpha$, we have
\begin{equation}
\kappa_1(b,C) = \BigO(e^{-\tilde C e^{b^p}}), \qquad b \to + \infty,
\end{equation}
with some $\tilde C>0$, \ie~a much faster decay than in $\kappa_2$. 

Notice that the set of suitable curves along which $\la$ can tend to infinity, \ie~the restrictions on $\beta$, see \eqref{Ex.exp.beta.2}, is substantially enlarged for $p \geq 1$ comparing to \eqref{Ex.exp.beta.1}.

\end{example}

\begin{example}(Logarithmic dampings and potentials -- continued).

Finally, we consider non-negative $a,q\in C^{n+1}(\R)$ with $n>1$ of the form 
\begin{equation}\label{a.q.log.exp.ex.2}
\forall x \gs 1, \quad a(x)=\ln(x), \quad |q^{(j)}(x)| \ls x^{-j} \ln x, \quad j \in [[0,n]],
\end{equation}
which satisfies Assumption~\ref{Asm:exp} with $\nu =-1$. Similarly as in Example~\ref{Ex:log.1}, the conditions~\eqref{kappa1.def}, \eqref{Thm:exp.asm.q.0}, \eqref{Thm:exp.asm.q} and \eqref{Thm:exp.asm.cutoff-2} hold without any restriction on $\beta = \beta(b)>0$. If $\beta$ is not too small, \eg~$\beta(b) \gs b^s$ with $s >-1$, the term $\kappa_1(b,C)$ exhibits much faster decay (exponential in $b$) than $\kappa_2$, for which we obtain
\begin{equation}
\kappa_2(b) = \BigO((\ln b)^{1-n} b^{-n-1}), \qquad b \to + \infty.
\end{equation}

\end{example}

{\footnotesize
	\bibliographystyle{acm}
	\bibliography{C:/Data/00Synchronized/references}
}

\end{document}